\documentclass[pdflatex,sn-mathphys-num]{sn-jnl}

\usepackage{graphicx}%
\usepackage{multirow}%
\usepackage{amsmath,amssymb,amsfonts}%
\usepackage{amsthm}%
\usepackage{mathrsfs}%
\usepackage[title]{appendix}%
\usepackage{xcolor}%
\usepackage{textcomp}%
\usepackage{manyfoot}%
\usepackage{booktabs}%
\usepackage{algorithm}%
\usepackage{algorithmicx}%
\usepackage{algpseudocode}%
\usepackage{listings}%

\usepackage{amsmath}
\usepackage{color}
\usepackage{amsfonts}
\usepackage{amscd}
\usepackage{latexsym}
\usepackage{amssymb}  
\usepackage{amsthm}
\usepackage{hyperref}
\allowdisplaybreaks
\usepackage{graphicx, graphics}
\usepackage{subfig}
\usepackage{subcaption}

\numberwithin{equation}{section}
\newtheorem{theorem}{Theorem}[section]
\newtheorem{lemma}{Lemma}[section]

\newtheorem{definition}{Definition}[section]

\newcommand{\bu}{{u}}
\newcommand{\bv}{{v}}
\newcommand{\bw}{{w}}

\newcommand{\but}{\bu(\cdot,t)}
\newcommand{\bvt}{\bv(\cdot,t)}
\newcommand{\bwt}{\bw(\cdot,t)}

\newcommand{\bus}{\bu(\cdot,s)}
\newcommand{\bvs}{\bv(\cdot,s)}
\newcommand{\bws}{\bw(\cdot,s)}

\newcommand{\gu}{{\nabla \bu}}
\newcommand{\gv}{{\nabla \bv}}
\newcommand{\gw}{{\nabla \bw}}

\newcommand{\gvs}{{\nabla \bvs}}
\newcommand{\gws}{{\nabla \bws}}

\newcommand{\lv}{{\Delta \bv}}
\newcommand{\lw}{{\Delta \bw}}

\newcommand{\mlv}{{|\Delta \bv|}}
\newcommand{\mlw}{{|\Delta \bw|}}
\newcommand{\mgufpt}{\big|\gu^{\frac{p}{2}}\big|}

\newcommand{\buk}{\bu^k}
\newcommand{\bul}{\bu^l}
\newcommand{\bum}{\bu^m}

\newcommand{\bums}{\bu^m(\cdot,s)}
\newcommand{\buks}{\bu^k(\cdot,s)}

\newcommand{\rsn}{\mathbb{R}^n}
\newcommand{\lis}{{ \mathcal{L}^{\infty}}(\Omega)}
\newcommand{\los}{{\mathcal{L}^{1}}(\Omega)}
\newcommand{\lts}{{\mathcal{L}^{2}}(\Omega)}
\newcommand{\ltps}{{\mathcal{L}^{\frac{2}{p}}}(\Omega)}
\newcommand{\lps}{{\mathcal{L}^{p}}(\Omega)}
\newcommand{\lqs}{{ \mathcal{L}^{q}}(\Omega)}
\newcommand{\lrs}{{\mathcal{L}^{q_0}}(\Omega)}
\newcommand{\lpos}{{ \mathcal{L}^{p^{'}}}(\Omega)}
\newcommand{\lpms}{{ \mathcal{L}^{p+m-1}}(\Omega)}
\newcommand{\ltpps}{{\mathcal{L}^{\frac{2p'}{p}}}(\Omega)}

\newcommand{\wsq}{{\mathcal{ W}^{1,q}}(\Omega)}
\newcommand{\cso}{ { \mathcal{C}^{0}(\overline{\Omega})}}
\newcommand{\cts}{ { \mathcal{C}^{2}(\overline{\Omega})}}

\newcommand{\wsqob}{{ \mathcal{W}^{1,q}(\overline{\Omega})}}

\newcommand{\lros}{{ \mathcal{L}^{q_0\sigma }(\Omega)}}

\newcommand{\nru}{\Big\|\bu\Big\|}

\newcommand{\nut}{\big\|\but\big\|}
\newcommand{\nvt}{\big\|\bvt\big\|}
\newcommand{\nwt}{\big\|\bwt\big\|}

\newcommand{\ngvt}{\big\|\nabla\bvt\big\|}

\newcommand{\ngufpt}{\Big\|\gu^{\frac{p}{2}}\Big\|}
\newcommand{\nufpt}{\Big\|\bu^{\frac{p}{2}}\Big\|}

\newcommand{\ssup}{\sup\limits_{t\in(s_0,\tmax)}}
\newcommand{\tmax}{T_{\mathrm{max}}}

\newcommand{\evtsl}{e^{d_2 (t-s) \Delta}}

\newcommand{\eats}{e^{-\alpha (t-s)}}

\newcommand{\eutsl}{e^{d_1 (t-s) \Delta}}
\newcommand{\enpkts}{e^{-\left(\fnpk\right)(t-s)}}

\newcommand{\intt}{\int^t_0}
\newcommand{\intts}{\int^t_{s_0}}

\newcommand{\intT}{\int_{s_0}^T}
\newcommand{\inti}{\int_0^\infty}
\newcommand{\ints}{\int_{\Omega}}
\newcommand{\ds}{\mathrm{d}s}
\newcommand{\dx}{\mathrm{d}x}
\newcommand{\dt}{\frac{\mathrm{d}}{\mathrm{d}t}}

\newcommand{\up}{{\bu}^p}
\newcommand{\upo}{{\bu}^{p-1}}
\newcommand{\upt}{{\bu}^{p-2}}

\newcommand{\upk}{{\bu}^{p+k-1}}

\newcommand{\upm}{{\bu}^{p+m-1}}

\newcommand{\fpk}{\frac{p-1}{p+k-1}}
\newcommand{\fnpk}{\frac{np+2}{2-(k-1)n}}
\newcommand{\fnpt}{\frac{np+2}{n}}

\newcommand{\tgpmp}{\frac{2\gamma(p+m-1)}{p}}
\newcommand{\togpmp}{\frac{2(1-\gamma)(p+m-1)}{p}}

\begin{document}

\title[Existence of attraction-repulsion chemotaxis system with nonlocal source]{Global existence and boundedness in an attraction-repulsion chemotaxis system with nonlocal logistic source and sublinear productions}

\author*[1,2]{\fnm{Gnanasekaran} \sur{shanmugasundaram}}\email{dr.sakar.mat@gmail.com}
\author[1]{\fnm{Nithyadevi} \sur{nagarajan}}\email{nithyadevin@buc.edu.in}

\affil*[1]{\orgname{Bharathiar University}, \city{Coimbatore}, \postcode{641046}, \state{Tamilnadu}, \country{India}}

\affil[2]{\orgname{National Institute of Technology}, \city{Tiruchirappalli}, \postcode{620015}, \state{Tamilnadu}, \country{India}}

%
%
%
%
%

\abstract{
This paper deals with the following attraction-repulsion chemotaxis system with  nonlocal logistic source and sublinear productions
\begin{align*}
	\left\{
	\begin{array}{rrll}
	&&\bu_t=d_1 \Delta \bu-\chi \nabla\cdot(\buk \gv)+\xi \nabla\cdot(\buk \gw)+\displaystyle{\mu \bum\left(1-\ints\bu(x,t)\dx\right)},\hspace*{0.5cm} &x\in\Omega,\, t>0,\\
	&&\bv_t=d_2 \Delta \bv-\alpha\bv+f(u), &x\in\Omega,\, t>0,\\
	&&\bw_t=d_3 \Delta \bw-\beta\bw+f(u), &x\in\Omega,\, t>0,\\
	&&\frac{\partial \bu}{\partial\nu}=\frac{\partial \bv}{\partial\nu}=\frac{\partial \bw}{\partial\nu}=0, &x\in\partial\Omega,\, t>0,\\
	&&\bu(x,0)=\bu_0, \quad \bv(x,0)=\bv_0, \quad \bw(x,0)=\bw_0,&x\in\Omega,
	\end{array}
\right.
\end{align*}
in a bounded domain $\Omega\subset\mathbb{R}^n$, $n\geq 2$ with smooth boundary $\partial\Omega$. Assume the parameters $d_1$, $d_2$, $d_3$, $\chi$, $\xi$, $\alpha$, $\beta$ and $\mu$ are positive constants and $k, m\geq 1$. The initial data $(\bu_0, \bv_0, \bw_0)$ are nonnegative and the function $f(u)\leq K u^l\in C^1([0, \infty))$ for some $K, l>0$. Under appropriate conditions on the parameter $k$, $l$ and $m$ we show that the above problem admits a unique globally bounded classical solution.}



\maketitle

\section{Introduction and motivation}
\quad Many researchers are currently involved in the analysis of nonlinear partial differential equations considering the nonlocal source. In recent years, works on dynamical systems with nonlocal source  are gaining much attention in mathematical biology \cite{sbian2017, sbian}. In any biological phenomenon, in general, the cells interact not only through attraction, there exist cases of repulsion between them as well. This necessitates the study of attraction-repulsion mechanism in the context of chemotaxis. Microglia are neuronal support cells in the central nervous system and they interact with secretory signals through chemotaxis to secrete attractive and repulsive chemicals. In the case of Alzheimer's disease (a type of neurodegenerative disease), the motion of microglias is affected by such chemoattractants and chemorepellents produced by them, which stimulate its aggregation in the central nervous system. This serves as a good motivation to study chemotaxis system involving one species with two chemicals. The process was modelled using the attraction-repulsion Keller-Segel system as in \cite{mluca} and mathematical studies of similar models have been widely explored thereafter.

\quad This study inherits the following initial-boundary value problem
{\small
	\begin{align}
		\left\{
		\begin{array}{rrll}
			&&\bu_t=d_1 \Delta \bu-\chi \nabla\cdot(\buk \gv)+\xi \nabla\cdot(\buk \gw)+\displaystyle{\mu \bum\left(1-\ints\bu(x,t)\dx\right)},\hspace*{0.1cm} &x\in\Omega,\, t>0,\\
			&&\bv_t=d_2 \Delta \bv-\alpha\bv+f(u), &x\in\Omega,\, t>0,\\
			&&\bw_t=d_3 \Delta \bw-\beta\bw+f(u), &x\in\Omega,\, t>0,\\
			&&\frac{\partial \bu}{\partial\nu}=\frac{\partial \bv}{\partial\nu}=\frac{\partial \bw}{\partial\nu}=0, &x\in\partial\Omega,\, t>0,\\
			&&\bu(x,0)=\bu_0, \quad \bv(x,0)=\bv_0, \quad \bw(x,0)=\bw_0,&x\in\Omega,
		\end{array}
		\right.\label{2}
\end{align}}
in a bounded domain $\Omega\subset\mathbb{R}^n$, $n\geq 2$ with smooth boundary $\partial\Omega$ and here $\nu$ denotes the outward unit normal on $\partial\Omega$. We assume the parameters $d_1$, $d_2$, $d_3$, $\chi$, $\xi$, $\alpha$, $\beta$, $\mu$ are positive constants, $k, m\geq 1$ and the initial data are sufficiently regular. The unknown functions $\bu$, $\bv$ and $\bw$ represent the cell density, the concentration of attractive and repulsive chemical substances, respectively. The system \eqref{2} describes the interaction of one cell with two different chemicals secreted by the cells. The attraction effect is given by the term $-\chi\nabla\cdot(u^k\nabla v)$, where $\chi$ is the attraction coefficient. On the other hand, the term $\xi\nabla\cdot(u^k\nabla w)$ represent the repulsion effect, where $\xi$ is the repulsion coefficient. The proliferation rate of cells are assumed to have a logistic growth in the form of nonlocal source and is given by $\mu u^m(1-\int_{\Omega} u)$, where $\mu$ is the growth rate constant.  Here, the decay rates of the chemicals are given by $\alpha v$ and $\beta w$ respectively. Here the growth rate of the cells is named by the function $f(u)$. Further, $d_1$, $d_2$ and $d_3$ are positive diffusion coefficients of cell density and chemical signals respectively.

\quad In addition, we also assume that the initial data $\bu_0$, $\bv_0$ and $\bw_0$ satisfy
\begin{align}
	\left\{
	\begin{array}{rrll}
		&&\bu_0\in\cso,\quad \mbox{with} \quad \bu_0 \geq 0\quad\mbox{in}\,\, \Omega,\\
		&&\bv_0, \bw_0\in\wsqob,\quad \mbox{for some}\,\, { q} >n,\quad \mbox{with} \quad \bv_0, \bw_0 \geq 0\quad\mbox{in}\,\, \Omega.
	\end{array}
	\right.\label{3}
\end{align}
Moreover, the function $f$ is defined as follows
\begin{align}
	f(s)\in C^1([0, \infty)) \quad \mbox{and} \quad 0\leq f(s)\leq K s^l \quad \mbox{for some} \quad K, l >0 \quad \mbox{and all} \quad s\geq 0.\label{3.1}
\end{align}

\quad The system \eqref{2} is the generalization of the mimimal chemotaxis one species and one stimuli system by Keller-Segel in 1970.
\begin{align*}
	\left\{
	\begin{array}{rrll}
		\bu_t&=&\Delta \bu-\nabla\cdot(\bu \nabla \bv),\\
		\bv_t&=&\Delta \bv-\bv+\bu,
	\end{array}
	\right.
\end{align*}
where $\bu$ and $\bv$ denote the cell density and the concentration of chemical signals, respectively. Analytic and numerical simulations of various Keller-Segel models and Keller-Segel models with logistic source have been studied in the past few decades. For the recent evolution in this field one can refer the articles by Lankeit and Winkler \cite{lan}, Bellomo et al. \cite{nbellomo} , Horstmann \cite{horstmann} and the references therein.

\quad To the better understanding of the system \eqref{2}, Luca \cite{mluca} proposed the attraction-repulsion system
\begin{align}
	\left\{
	\begin{array}{rrll}
		&&\bu_t= \nabla\cdot(D(\bu) \nabla \bu)-\chi \nabla\cdot(\bu \gv)+\xi \nabla\cdot(\bu \gw),\\
		&&\tau_1\bv_t=d_1 \Delta \bv+\alpha\bu-\beta\bv,\\
		&&\tau_2\bw_t=d_2 \Delta \bw+\gamma\bu-\delta\bw.
	\end{array}
	\right.\label{5}
\end{align}
For the case $\tau_1=1$ and $\tau_2=0$, Lin et. al. \cite{klin} proved the existence and global bounded solution with $D(\bu)=D_0\bu^{-\theta}$, under the condition that $\theta<\frac{2}{n}-1$. By using Lyapunov functional, the finite time blow up is also studied for the radially symmetric solutions with $\theta=\frac{2}{n}-1$ and $n=3$. In addition, blow-up of the solution to the sytem \eqref{5} may exist, if $\theta>\frac{2}{n}-1$ and $\chi\alpha-\xi\gamma>0$ with $n\geq 3$. On the contrary, if $\xi\gamma-\chi\alpha>0$ or $m>2-\frac{2}{n}$, Wang \cite{ywang2016} discussed the global bounded solutions to the system \eqref{5} for the non-degenerate diffusion with $D(\bu)\geq C_D \bu^{m-1}$, $m\geq 1$. Additionally, the global bounded weak solution for the case of degenerate diffusion is proved under the same condition. If $\chi\alpha-\xi\gamma<0$, global bounded classical solution to the system \eqref{5} has been studied by Guo et. al. \cite{qguo}, while, if $\chi\alpha-\xi\gamma>0$, the solution of the system is bounded whenever $\|\bu_0\|_{\los}<\frac{4\pi}{\chi\alpha-\xi\gamma}$. On the other hand, if $\|\bu_0\|_{\los}>\frac{4\pi}{\chi\alpha-\xi\gamma}$, the solution may blow-up.

\quad For the fully parabolic case, Jin et. al. \cite{jinwang} improved the conditions of Liu et. al. \cite{liuwang} in one space dimension. They also proved that, when the attractive and repulsive chemical signals have the same degradation with attraction dominating the repulsion, the solution of the system \eqref{5} converges as $t\to \infty$, algebraically. Conversely, Jin \cite{hyjin} showed the existence global bounded classical solution when the repulsion dominates the attraction in two space dimension. Further, they proved that the weak solution in three dimensional settings was discussed with large initial data. Based on the entropy-like inequality and coupled estimate techniques for global bounded classical solutions in two dimensional case, see Liu et. al. \cite{liutao}. If $\xi\gamma=\chi\alpha$, the existence of global solutions to the system \eqref{5} as well as the asymptotic stability results were presented by Jin et. al. \cite{jinliu} for the case $n=2,3$. An interesting fact to note here that, this solution converge like the heat kernel as $t\to \infty$. Whenever $m>1$, where $D(\bu)\geq C(\bu+\epsilon)^{m-1}$ with $\epsilon>0$, the global bounded classical solutions is examined by Wang et. al. \cite{wangxiang}. They also proved that, the weak solution exists for the case $\epsilon=0$ under the same assumptions. If $\|\bu_0\|_{\los}<\frac{1}{C(\Omega)\chi\alpha}$, the existence of global bounded classical solutions and the asymptotic behavior to the system \eqref{5} is given by Lin et. al. \cite{linmu} based on the entropy type inequality.  In two dimensional case, Jin et. al. \cite{hyjinwang} proved the global existence of classical solutions under some conditions on the parameters and provide the asymptotic results with large initial data. In addition, if $\frac{\xi\gamma}{\chi\alpha}>\max\{\frac{\beta}{\delta}, \frac{\delta}{\beta}\}$, the solution of the system \eqref{5} converge in terms of exponential. For more details, see \cite{liwang, swu, xuzheng, hyi2025}.

\quad An attraction-repulsion chemotaxis system with logistic source is given by
\begin{align}
	\left\{
	\begin{array}{rrll}
		&&\bu_t= \nabla\cdot(D(\bu) \nabla \bu)-\chi \nabla\cdot(\bu \gv)+\xi \nabla\cdot(\bu \gw)+f(\bu),\\
		&&\tau\bv_t=d_1 \Delta \bv+\alpha\bu-\beta\bv,\\
		&&\tau\bw_t=d_2 \Delta \bw+\gamma\bu-\delta\bw.
	\end{array}
	\right.\label{6}
\end{align}
For the parabolic-elliptic case, where $D(\bu)\geq C_D(\bu+\zeta)^{m-1}$, $m\geq 1$ and $f(\bu)\leq a-b \bu^\eta$, $\eta>1$ Wang \cite{ywang2015} showed the existence of global bounded classical solutions to the system \eqref{6} for the case $\zeta>0$, and also, when either $\xi\gamma-\chi\alpha>0$ or logistic damping is large or diffusion is sufficiently large enough. Further, the system has weak solutions for the case $\zeta=0$ under the same assumptions. In addition, the asymptotic behavior has been studied to the system \eqref{6} with particular logistic source. Under some suitable conditions on the logistic source and the parameters, Zhang et. al. \cite{qzhang} obtained the global bounded classical solution as well as the weak solution to the system \eqref{6} with $D(\bu)\geq -C \bu(1+\bu)$ and the large time behavior for $f(\bu)=\mu\bu(1-\bu)$ respectively. Zhao et. al. \cite{jzhao} investigated the global bounded solution when $f(\bu)= r\bu-\mu \bu^2$ under the condition that $\chi\alpha-\xi\gamma>0$ and $\mu=\frac{1}{3}(\chi\alpha-\xi\gamma)$ with $D(u)\geq C_D u^{m-1}$, $m\geq 1$. Further, the exponentially convergence result is provided for the case when $D(\bu)=1$ and $n\geq 3$. On the other hand, they studied the boundedness and blow-up for the case when $D(\bu)=0$ and $n\geq 2$.  When $D(\bu)\geq C_D(\bu+1)^{m-1}$, $m\geq 1$ and $f(\bu)\leq a-b \bu^\eta$, $\eta>1$, Yan et. al. \cite{lyan} proved the global bounded solution to the system \eqref{6} and this solution may blow-up whenever $\|\bu_0\|_{\los}>\frac{8\pi}{\chi\alpha-\xi\gamma}$ and $\chi\alpha-\xi\gamma>0$ under the specific functions in two dimensional settings. For the recent article see, \cite{wdu2024}.

\quad Li et. al. \cite{lixiang} proved the system \eqref{6} has global bounded solutions for both the cases $\tau=0,1$ with $f(\bu)\leq a-b\bu^{\rho}$. For the case $\tau=0$ and $\xi\gamma-\chi\alpha>0$, there exists a bounded classical solution to the system \eqref{6} whenever $\rho\geq 1$. On the other hand, if $\xi\gamma-\chi\alpha<0$, still the result is valid whenever $\rho$ is strong enough. In one and two space dimensions, the author established the same results when $\tau>0$.  For the fully parabolic case, Zheng et. al. \cite{zhengmu} proved the existence of global bounded classical solution to the system \eqref{6} with $f(\bu)\leq a\bu-\mu\bu^2$ under some suitable conditions and this result improved the result of Li et. al. \cite{lixiang} provided $n\geq 3$ and $\beta=\delta$. In three dimensional settings, for $\beta,\delta\geq \frac{1}{2}$ and $\mu\geq \max\{(\frac{41}{2}\chi\alpha+9\xi\gamma)^{\eta},(9\chi\alpha+\frac{41}{2}\xi\gamma)^{\eta}\}$, the global bounded classical solutions to \eqref{6} with $f(\bu)=\bu-\mu\bu^{\eta+1}$, $\eta>1$ has been studied by Li et. al. \cite{limu} and this solution converges to the steady state as $t\to\infty$. If, $\chi\alpha=\xi\gamma$, Shi et. al. \cite{shiliu} established the global bounded classical solutions to the system \eqref{6} with $f(\bu)\leq a-b\bu^\theta$ under $\theta$ and $n$ satisfying some suitable conditions. If $\mu> \frac{\chi^2\alpha^2(\beta-\delta)^2}{8\delta\beta^2}$, the exponential stability is discussed with $f(\bu)=\mu\bu(1-\bu)$ for $n\leq 9$ by constructing appropriate Lyapunov functional. Suppose $D(\bu)\geq C \bu^{m-1}$ and also $m>\frac{6}{5}$, the global bounded classical solution is obtained by Zeng \cite{yzeng}. On the other hand, same result holds for $D(\bu)\leq C(\bu^{m-1}+1)$ when $m\in(1, \frac{6}{5}]$. For more recent articles, one can refer \cite{chiyo, hetian, renliu}.

\quad Attraction-repulsion chemotaxis system with nonlinear signal production is studied with $f(u)\leq u(a-bu^s)$, $f(0)\geq 0$ by Hong et. al.  \cite{lhong}.
\begin{align}
	\left\{
	\begin{array}{rrll}
		&&\bu_t= \Delta \bu-\chi \nabla\cdot(\bu \gv)+\xi \nabla\cdot(\bu \gw)+f(\bu),\\
		&&0= \Delta \bv+\alpha\bu^k-\beta\bv,\\
		&&0=\Delta \bw+\gamma\bu^l-\delta\bw.
	\end{array}
	\right.\label{6.1}
\end{align}
The existence of global solution is proved under some conditions on the attraction with max$\{l,s,\frac{2}{n}\}>k$. They also showed that the size of the coefficients will affect the boundedness of the solution and these results improved the results of \cite{qzhang}. Liu et. al. \cite{mliu} established the finite time blowup results to the following system with nonlinear production
\begin{align}
	\left\{
	\begin{array}{rrll}
		&&\bu_t= \Delta \bu-\chi \nabla\cdot(\bu \gv)+\xi \nabla\cdot(\bu \gw),\\
		&&0= \Delta \bv-\mu_1(t)+f_1(u),\\
		&&0=\Delta \bw-\mu_2(t)+f_2(u),
	\end{array}
	\right.\label{6.2}
\end{align}
where, $f_1(s)=s^{\gamma_1}$, $f_2(s)=s^{\gamma_2}$, $\mu_1(t)=\frac{1}{|\Omega|}\int f_1(u)$ and $\mu_2(t)=\frac{1}{|\Omega|}\int f_2(u)$. The system \eqref{6.2} blows up in finite time when $\gamma_1>\gamma_2$ and $\gamma_1>\frac{2}{n}$. In addition, global bounded solution is proved, for $\gamma_1<\frac{2}{n}$ with suitable initial data. The behaviour of the solution for the case $\gamma_1>\frac{2}{n}$ and $\gamma_1<\gamma_2$ is still an open question. Zhou et. al. \cite{xzhou} improved the condition for the global bounded solution to the system \eqref{6.1} which has already been proved by Hong et. al. \cite{lhong} and also established the asymptotic behavior of the global solution to the system. The large time behavior of attraction repulsion system with nonlinear secretion was studied by Ren et. al. \cite{gren}. Here, the convergence rate of the solution is either exponential or algebraical.

\quad For the study of nonlocal terms in the chemotaxis system, recently, Negreanu et al. \cite{mnegreanu2021} studied the following chemotaxis system with a nonlocal source
\begin{align*}
\left\{
\begin{array}{rrll}
	&&u_t= \Delta u-\chi \nabla\cdot (u^m\nabla v)+u(a_0-a_1u^\alpha+a_2\ints u^\alpha\dx),\\
	&&v_t=\Delta v-v+u^\gamma,
		\end{array}
	\right.
\end{align*}
If $\alpha\geq 1$, $m>1$, $\gamma\geq 1$, $\alpha+1>m+\gamma$, $a_1>0$, $a_1-a_2|\Omega|>0$ and under some suitable conditions on the initial data, the authors proved the global existence of solutions to the considered system. Moreover, the solution converges to the stead state $\displaystyle{u^*=\frac{a_0^\frac{1}{\alpha}}{(a_1-a_2|\Omega|)^\frac{1}{\alpha}}, v^*=(u^*)^\gamma}$. 

\quad Chiyo et al. \cite{chiyo2024} investigated the following system
\begin{align*}
\left\{
\begin{array}{rrll}
	&&u_t= \Delta u-\chi \nabla\cdot (u\nabla v)+au^\alpha-bu^\alpha\ints u^\beta,\\
	&&v_t=\Delta v-v+u,
		\end{array}
\right.
\end{align*}
for $\alpha, \beta\geq 1$. The authors established the uniform boundedness of classical solutions under two key regimes: the subquadratic case, $1\leq \alpha< 2$ and $\beta>\frac{n+4}{2}-\alpha$; the superquadratic case, $\beta>\frac{n}{2}, 2\leq \alpha<1+\frac{2\beta}{n}$. This work extends the analysis of the fully parabolic case previously studied in  \cite{bian2018}.

\quad Fuentes et al. \cite{funtes2025} studied the following system
\begin{align*}
\left\{
\begin{array}{rrll}
	&&u_t=\nabla\cdot((u+1)^{m_1-1} \nabla u-\chi  u(u+1)^{m_2-1}\nabla v)+B(u,\nabla v),\\
	&&v_t=\Delta v-v+f(u),
\end{array}
	\right.
\end{align*}
where the term $B(u,\nabla u)$ takes one of the following forms $au^\alpha-bu^\beta-c\ints u^\delta$ or $au^\alpha-bu^\alpha\ints u^\beta-c|\nabla u|^\delta$. For each case, the authors derived sufficient conditions on the system parameters that ensure the existence of globally bounded classical solutions. This work extends previous analyses from \cite{chiyo2024, bian2018}.

\quad Ren \cite{gren2022} studied the existence of global bounded classical solutions to a class of attraction-repulsion chemotaxis systems with nonlocal terms. Hu \cite{rhu2024} discussed the global existence of classical solutions to a two-species attraction-repulsion chemotaxis system with nonlocal terms and additionally studied the asymptotic stability under suitable Lyapunov functionals. In related work on nonlocal terms in chemotaxis systems, Columbu et al. \cite{columbu2024} established global boundedness results for classical solutions of an attraction-repulsion chemotaxis model. Further developments can be found in \cite{columbu20242}. 

\quad To the best of our knowledge, no prior studies have investigated the attraction-repulsion chemotaxis system incorporating a nonlocal logistic source with sublinear production terms. Motivated by this gap, we introduce a nonlocal source term into the attraction-repulsion framework while accounting for sublinear production kinetics. Our analysis examines how these modifications influence the existence of solutions and their long-term dynamics.

\quad Motivated by the above research works, main purpose of this paper is to establish the existence of the global classical solution to the system \eqref{2} with suitable conditions on $k$, $l$ and $m$, which is uniformly bounded in $\Omega\times(0, \infty)$ for $n\geq 2$. More to the point, this study is to examine the simultaneous effects of both attraction and repulsion mechanism in the presence of nonlocal source and sublinear production terms. 

\quad Our article is organized as follows: Section 2 contains some basic inequalities, lemmas and the local existence of classical solution. Section 3 deals with boundedness and global existence of classical solution to the system \eqref{2}. Finally, the work ends with conclusion in Section 4.

\begin{theorem}\label{t1}
Suppose that $\Omega \subset\rsn\: (n\geq 2)$, is a bounded domain with smooth boundary and $q >n$. Moreover, the function $f$ fulfills \eqref{3.1} with $l+k<1+\frac{2}{n}$, where $k\geq 1$ and $1<m< 1+\frac{2}{n}$. Then for any initial data $(\bu_0, \bv_0,\bw_0)$ satisfy \eqref{3}, the system \eqref{2} possesses a
	unique classical solution	$(\bu,\bv,\bw)$ which is global and uniformly bounded in the sense that
	\begin{align*}
		\nut_{\lis}+\nvt_{\wsq}+\nwt_{\wsq}\leq C, \qquad \forall\, t>0,
	\end{align*}
	where the constant $C$ is positive.
\end{theorem}


\section{Preliminaries and local existence}
\quad We recall some useful inequalities and key lemma which we are going to use in the sequel.  In this section, we begin with the local existence of solutions to system \eqref{2} which is standard and its proof is based on the ideas of \cite{horst}.
\begin{definition}[Cauchy's inequality with $\epsilon$ \cite{evans}]
\begin{align*}
ab\leq \epsilon a^2+\frac{1}{4\epsilon} b^2, \qquad\qquad a,b>0, \epsilon >0.
\end{align*}
\end{definition}
\begin{definition}[Young's inequality with $\epsilon$ \cite{evans}]
Let $1<p,\: q<\infty$, $\frac{1}{p}+\frac{1}{q}=1$. Then
\begin{align*}
ab\leq \epsilon a^p+ C(\epsilon) b^q, \qquad\qquad a,b>0, \epsilon>0,
\end{align*}
for $C(\epsilon)=(\epsilon p)^{\frac{-q}{p}} q^{-1}$.
\end{definition}
\begin{definition}[H\"older's inequality \cite{evans}]
Assume $1< p,\: q<\infty$, $\frac{1}{p}+\frac{1}{q}=1$. If $u\in\lps,\: v\in \lqs$, then we have
\begin{align*}
\ints|uv|\mathrm{d}x\leq \|u\|_{\lps}\:\|v\|_{\lqs}.
\end{align*}
\end{definition}
\begin{definition}[Interpolation inequality \cite{evans}]
	Assume $1\leq s\leq r \leq t \leq \infty$ and $\frac{1}{r}=\frac{\theta}{s}+\frac{(1-\theta)}{t}$, Suppose also $u\in{\mathcal{L}^{s}(\Omega)}\cap{\mathcal{L}^{t}(\Omega)}$. Then $u\in{\mathcal{L}^{r}(\Omega)}$ and
	\begin{align*}
		\big\|u\big\|_{\mathcal{L}^r(\Omega)}\leq \big\|u\big\|^{\theta}_{{ \mathcal{L}^s}(\Omega)}\:\big\|u\big\|^{1-\theta}_{\mathcal{L}^t(\Omega)}.
		\end{align*}
\end{definition}
\begin{definition}[The Gagliardo-Nirenberg inequality \cite{evans}]
	Fix $1 \leq q,r \leq \infty$ and a natural number $m$. Suppose also that a real number $\alpha$ and a natural number $j$ are such that\linebreak $\frac{1}{p}=\frac{j}{n}+\left(\frac{1}{r}-\frac{m}{n}\right)\alpha+\frac{1-\alpha}{q}$ and $\frac{j}{m}\leq \alpha \leq 1$,
	then
	\begin{align*}
		\big\|D^j u\big\|_{\lps}\leq C_1\big\|D^m u\big\|^\alpha_{{ \mathcal{L}^r}(\Omega)}\big\|u\big\|^{1-\alpha}_{\lqs}+C_2\big\|u\big\|_{{ \mathcal{L}^s}(\Omega)},
		\end{align*} 
	where $s > 0$ is arbitrary.
	\end{definition}

\begin{lemma}[Local Existence]\label{l1}
Suppose that $\Omega\subset\rsn\: (n\geq 2)$, is a bounded domain with smooth boundary and assume that $q>n$. Moreover, the function $f$ fulfills \eqref{3.1} and $k,m\geq 1$, then for each non-negative initial data satisfies \eqref{3}, there exists $\tmax\in (0,\infty]$ such that a uniquely determined triple $(\bu, \bv, \bw)$ of nonnegative functions
\begin{align*}
	\bu&\in { \mathcal{C}^{0}}\left(\overline{\Omega}\times\left.\left[0,\tmax\right.\right)\right)\cap { \mathcal{C}^{2,1}}\left(\overline{\Omega}\times\left(0,\tmax\right)\right),\nonumber\\
	\bv, \bw&\in { \mathcal{C}^{0}}\left(\overline{\Omega}\times\left.\left[0,\tmax\right.\right)\right)\cap { \mathcal{C}^{2,1}}\left(\overline{\Omega}\times\left(0,\tmax\right)\right)\cap { \mathcal{L}^{\infty}_{loc}}\left(\left.\left[0,\tmax\right.\right);\wsq\right),\nonumber
\end{align*}
solving \eqref{2} classically in $\Omega\times (0,\tmax)$. Furthermore, if $\tmax<\infty$, then
\begin{align}
	\lim_{t\to \tmax}\Big(\nut_{\lis}+\nvt_{\wsq}+\nwt_{\wsq}\Big)= \infty.\label{l1.1}
\end{align}
\end{lemma}
\begin{proof}
The proof can be derived by the standard arguments involving the parabolic regularity theory.  For the proof we refer to \cite{horst} and \cite{frassu}. 
In addition, the non-negativity of the solution in $\Omega\times(0, \tmax)$ comes from the maximum principle along with \eqref{3}.
\end{proof}

\begin{lemma}\label{l0}
(See \cite{sbian})	For $\ints \bu(x,0) \dx=\ints u_0(x) \dx=M_0>0$, the mass $\ints\bu(x,t)\dx=M(t)$ satisfies
\begin{align*}
	\min\{1,M_0\}\leq M(t)\leq \max\{1,M_0\}.
\end{align*}
Furthermore, we have the following decay estimates
\begin{align*}
	|1-M(t)|\leq |1-M_0|e^{-\min\big\{1,M_0^m\big\}t}.
\end{align*}
\end{lemma}

\begin{lemma}[Maximal Sobolev regularity \cite{xcao, hieber}]\label{l2}
Let $r\in(1,\infty)$ and $T\in (0, \infty)$. Consider the following evolution equation
\begin{align*}
	\left\{
	\begin{array}{rrll}
		&&y_t=\Delta y- y+g,\hspace*{2cm}& x\in \Omega,\: t>0,\\
		&&\frac{\partial y}{\partial\nu}=0, &x\in \partial\Omega,\: t>0,\\
		&&y(x, 0)=y_0(x), & x\in\Omega.
	\end{array}
	\right.
\end{align*}
For each $y_0\in{ \mathcal{W}^{2,r}}(\Omega)$ such that $\frac{\partial y_0(x)}{\partial\nu}=0$ on $\partial\Omega$ and any $g\in { \mathcal{L}^{r}}\big((0,T);{ \mathcal{L}^{r}}(\Omega)\big)$, there exists a unique solution
\begin{align*}
	y \in { \mathcal{W}^{1,r}}\big((0,T);{ \mathcal{L}^r}(\Omega)\big)\cap{  \mathcal{L}^r}\big((0,T); { \mathcal{W}^{2,r}}(\Omega)\big).
\end{align*}
Moreover, there exists $C_r>0$ such that
\begin{align*}
	\int_0^T\ints |y|^r+\int_0^T\ints |y_t|^r+\int_0^T\ints |\Delta y|^r\leq C_r\int_0^T\ints |g|^r+C_r\ints |y_0|^r+C_r\ints |\Delta y_0|^r.
\end{align*}
If $s_0\in(0,T)$ and $y(\cdot,s_0)\in{ \mathcal{W}^{2,r}}(\Omega)$ with $\frac{\partial y(\cdot,s_0)}{\partial\nu}=0$ on $\partial\Omega$ we have
\begin{align*}
	\intT\ints e^{sr}|\Delta y|^r\leq C'_r\intT\ints e^{sr}|g|^r+C'_r\ints |y(\cdot,s_0)|^r+C'_r\ints |\Delta y(\cdot,s_0)|^r,
\end{align*}
where $C_r'>0$.
\end{lemma}

\begin{lemma}[Extensibility criterion]\label{l5}
Let $n\geq 2, s_0\in (0, \tmax)$ with $s_0<1$ and the initial data $\bu_0$, $\bv_0$ and $\bw_0$ satisfy \eqref{3} for some $q >n$. Moreover, the function f fulfill \eqref{3.1} with $ l+k<1+\frac{2}{n}$, where $k\geq 1$ and $1<m< 1+\frac{2}{n}$. Suppose that there exists $p>\frac{n}{2}\geq 1$ such that
\begin{align*}
	\ssup\nut_{\lps}< \infty,
\end{align*}
then, we have
\begin{align*}
	\sup_{t>0}\Big(\nut_{\lis}+\nvt_{\wsq}+\nwt_{\wsq}\Big)< \infty.
\end{align*}
\end{lemma}
\begin{proof}
Let $q>n$ and for each fixed $p>\frac{n}{2}$ there holds
\begin{align*}
	\frac{np}{(n-p)_+}=
	\left\{
	\begin{array}{rrll}
		\hspace*{-2cm}&&\infty,  &\quad \mbox{if}\quad p\geq n,\\\\
		\hspace*{-2cm}&&\frac{np}{n-p}>n, & \quad \mbox{if}\quad \frac{n}{2}< p < n,
	\end{array}
	\right.
\end{align*}
and choose $q<\frac{np}{(n-p)_+}$ and $1<q_0<q$ fulfilling $n<q_0<\frac{np}{(n-p)_+}$
which enables to choose $\sigma>1$, \: $n < \sigma  q_0<\frac{np}{(n-p)_+}$ and $\sigma q_0 \leq q$. We fix arbitrary $s_0 \in (0, \tmax)$ with $s_0 < 1$ and $t\in (s_0, \tmax)$. Applying the variation of constants formula to the second equation of $\eqref{2}$, we get
\begin{align*}
	\bvt=e^{-\alpha t} e^{d_2t\Delta}\bv_0+\intts\eats\evtsl f\big(u(\cdot, s)\big)\ds.
\end{align*}
Now,
\begin{align*}
	\ngvt_{\lros}\leq&\: e^{-\alpha t}\big\|\nabla e^{d_2t\Delta}\bv_0\big\|_{\lros}\\
	&+ K\intts\eats\Big\|\nabla\evtsl\bul(\cdot,s)\Big\|_{\lros}\ds.
\end{align*}
By use of the Neumann heat semigroup \cite{winkler} with $q_0\sigma \leq q$, we attain
\begin{align}
	\ngvt_{\lros} \leq&\: C_1 e^{-\alpha t}\|\bv_0\|_{\wsq}\nonumber\\
	&+C_2\intts\eats\left(1+\left(t-s\right)^{-\frac{1}{2}-\frac{n}{2}\left(\frac{1}{p}-\frac{1}{q_0\sigma}\right)}\right)e^{-d_2\lambda(t-s)}\big\|\bul(\cdot,s)\big\|_{\lps}\ds,\label{l5.4}
\end{align}
where $C_1$ and $C_2$ are positive constants. Here, we can guarantee that $\frac{1}{2}+\frac{n}{2}\left(\frac{1}{p}-\frac{1}{q_0\sigma}\right) <1$, because of our assumption $q_0 \sigma<\frac{np}{(n-p)}$.
	Now, adopting the Gamma function, we ensure $\inti e^{-\alpha\psi}\left(1+\psi^{-\frac{1}{2}-\frac{n}{2}(\frac{1}{p}-\frac{1}{q_0\sigma})}\right)e^{-d_2\lambda\psi} < \infty$. Using the H\"older's inequality, to get
	\begin{align}
		\big\|\bul(\cdot, t)\big\|_{\lps}= \left(\int_\Omega u^{l p}(\cdot, t)\right)^\frac{l}{p}\leq  |\Omega|^{\frac{1-l}{p}}\left(\int_\Omega u^p(\cdot, t)\right)^\frac{l}{p}\leq C_3\|\bu(\cdot, t)\big\|^l_{\lps}, \quad \forall t\in (s_0, \tmax), \label{l5.5}
	\end{align}
	where $l\in(0,1)$ and $C_3>0$. 
	Substituting \eqref{l5.5} in to \eqref{l5.4}, we arrive at
	\begin{align*}
		\ngvt_{\lros}&\leq C_1\|\bv_0\|_{\wsq}+ C_4, \qquad \forall t\in (s_0, \tmax)
	\end{align*}
	where $C_4$ is positive constant.
	Finally, we obtain with the help of Lemma \ref{l1} and $s_0<1$
	\begin{align}
		\ngvt_{\lros}&\leq C_5, \qquad\qquad \forall\, t\in (0, \tmax).\label{l5.6}
	\end{align}
	Like wise we obtain,
	\begin{align}
		\big\|\gw(\cdot,t)\big\|_{\lros}&\leq C_6, \qquad\qquad \forall\, t\in (0, \tmax).\label{l5.6.1}
	\end{align}
	where $C_5$ and $C_6>0$. Set $t\in (0, \tmax)$ and we attain the following representation by adopting the variation of constants formula
	\begin{align*}
		\but=&e^{d_1t\Delta}\bu(\cdot,0)-\chi\intt\eutsl\nabla\cdot\Big(\buks\gvs\Big)\ds\nonumber\\
		&+\xi\intt\eutsl\nabla\cdot\Big(\buks\gws\Big)\ds\nonumber\\
		&+\mu\intt\eutsl\bums\Big(1-\ints\bus\Big)\ds.
	\end{align*}
	The above representation yields
	\begin{align*}
		\nut_{\lis}\leq &\: \Big\|e^{d_1t\Delta}\bu(\cdot,0)\Big\|_{\lis}+\chi\intt\Big\|\nabla\eutsl\Big(\buks\gvs\Big)\Big\|_{\lis}\ds\\
		&+\xi\intt\Big\|\nabla\eutsl\Big(\buks\gws\Big)\Big\|_{\lis}\ds\\
		&+\mu\intt\left\|\eutsl\bums\big(1-M(s)\big)\right\|_{\lis}\ds,
	\end{align*}
	for all $t\in(0,\tmax)$. Now, applying the Neumann heat semigroup \cite{winkler} with $C_7, C_8$ and $C_9 > 0$ satisfying
	\begin{align}
		\nut_{\lis}\leq &\:\big\|\bu(\cdot,0)\big\|_{\lis}\nonumber\\
		&+ C_7\intt\left(1+\left(t-s\right)^{-\frac{1}{2}-\frac{n}{2}\left(\frac{1}{q_0}-\frac{1}{\infty}\right)}\right)e^{-d_1\lambda(t-s)}\Big\|\buks\gvs\Big\|_{\lrs}\ds\nonumber\\
		&+ C_8\intt\left(1+\left(t-s\right)^{-\frac{1}{2}-\frac{n}{2}\left(\frac{1}{q_0}-\frac{1}{\infty}\right)}\right)e^{-d_1\lambda(t-s)}
		\Big\|\buks\gws\Big\|_{\lrs}\ds\nonumber\\
		&+|1-M_0|C_{9}\intt\left(1+\left(t-s\right)^{-\frac{n}{2}\left(\frac{1}{q_0}-\frac{1}{\infty}\right)}\right)e^{-d_1\lambda(t-s)}\Big\|\bums\Big\|_{\lrs}\ds,\label{l5.8}
	\end{align}
	where $\frac{1}{2}+\frac{n}{2q_0}<1$ because of $q_0>n$. Employing the Gamma function again,  the integral $\inti\left(1+\psi^{-\frac{1}{2}-\frac{n}{2q_0}}\right)e^{-\lambda\psi} < \infty.$ The H\"older's inequality and the Interpolation inequality gives us that
	\begin{align*}
		\big\|\buks\gvs\big\|_{\lrs} &\leq \big\|\buks\big\|_{ L^{q_0\widehat{\sigma}}(\Omega)}\; \big\|\gvs\big\|_{\lros}\nonumber\\
		&= \big\|\bus\big\|^k_{ L^{q_0\widehat{\sigma}k}(\Omega)}\; \big\|\gvs\big\|_{\lros}\nonumber\\
		&\leq \Big\|\bus\Big\|_{\los}^{k\zeta_2}\; \Big\|\bus\Big\|_{\lis}^{k(1-\zeta_2)}\; \big\|\gvs\big\|_{\lros}\nonumber\\
		&= \Big\|\bus\Big\|_{\los}^{\frac{1}{q_0\widehat{\sigma}}}\; \Big\|\bus\Big\|_{\lis}^{k-\frac{1}{q_0\widehat{\sigma}}}\; \big\|\gvs\big\|_{\lros}\nonumber\\
		&\leq C_{10}\Big\|\bus\Big\|_{\lis}^{k-\frac{1}{q_0\widehat{\sigma}}},
	\end{align*}
	where $C_{10}>0$. Now, using the Young's inequality with $q_0\widehat{\sigma}>\frac{n}{2}$ and $C_{11}>0$, gives
	\begin{align}
		\big\|\buks\gvs\big\|_{\lrs} 
		&\leq \frac{1}{6}\big\|\bus\big\|_{\lis}+C_{11}.\label{l5.9}
	\end{align}
	Similarly, we can attain with $C_{12}>0$
	\begin{align}
		\big\|\buks\gws\big\|_{\lrs}&\leq \frac{1}{6}\big\|\bus\big\|_{\lis}+C_{12},\label{l5.10}
	\end{align}
	where $\widehat{\sigma}$ is the dual exponent of $\sigma$ and $\displaystyle{\zeta_2=\frac{1}{q_0\widehat{\sigma}k} \in (0,1)}$, for all $s\in (0, \tmax)$. Also using the Interpolation inequality to the last term in \eqref{l5.8}, we get
	\begin{align*}
		\big\|\bums\big\|_{\lrs}=&\Big\|\bus\Big\|^m_{\mathcal{L}^{q_0 m}(\Omega)}\nonumber\\
		\leq& \Big\|\bus\Big\|^{m\zeta_3}_{\los}\:\Big\|\bus\Big\|^{m(1-\zeta_3)}_{\lis}\nonumber\\
		\leq& C_{13}\Big\|\bus\Big\|_{\lis}^{m-\frac{1}{q_0}},
	\end{align*}
	where $\displaystyle{\zeta_3=\frac{1}{q_0 m} \in (0,1)}$, for all $s\in (0, \tmax)$ and $C_{13}>0$. Once again the Young's inequality with $q_0>\frac{n}{2}$, gives
	\begin{align}
		\big\|\bums\big\|_{\lrs}
		\leq& C_{13}\Big\|\bus\Big\|_{\lis}^{m-\frac{1}{q_0}}\leq\frac{1}{6}\big\|\bus\big\|_{\lis}+C_{14},\label{l5.11}
	\end{align}
	where $C_{14}>0$. Inserting \eqref{l5.9}-\eqref{l5.11} in \eqref{l5.8}, it follows that
	\begin{align*}
		\sup_{t>0}\nut_{\lis}\leq \frac{1}{2}\sup_{t>0}\nut_{\lis}+C_{15},
	\end{align*}
	for all $t\in(0,\tmax)$, where $C_{15}>0$. Finally, we conclude that
	\begin{align}
		\nut_{\lis}\leq C_{16}, \qquad\qquad\qquad\forall \, t\in(0,\tmax).\label{l5.14}
	\end{align}
	where the constant $C_{16}$ is positive. Initial data regularity $v_0, w_0\in \wsq$
	and the smoothing property of the heat semigroup ensure $v, w\in \mathcal{L}^\infty((0, \tmax); \wsq)$. This will imply that actually 
	\begin{align*}
		\sup_{t>0}\Big(\nut_{\lis}+\nvt_{\wsq}+\nwt_{\wsq}\Big)< \infty.
	\end{align*}
	This completes the proof.
\end{proof}

\section{Global existence of classical solutions}
\quad In this section, we prove the global existence and boundedness of the solution to \eqref{2}. First we derive $\lps$ norm for $\bu$. Given any $s_0\in (0,\tmax)$ such that $s_0 < 1$, Lemma \ref{l1} gives  $\bu(\cdot,s_0), \bv(\cdot,s_0), \bw(\cdot,s_0)\in\cts$ with $\frac{\partial\bv(\cdot, s_0)}{\partial\nu}=0$ and $\frac{\partial\bw(\cdot, s_0)}{\partial\nu}=0$, we pick $C>0$ such that 
{\small
	\begin{align}
		\left\{
		\begin{array}{rrll}
			\hspace*{-0.5cm}&&\sup\limits_{0\leq s\leq s_0}\|\bus\|_{\lis}\leq C, \quad \sup\limits_{0\leq s\leq s_0}\|\bvs\|_{\lis}\leq C, \quad \sup\limits_{0\leq s\leq s_0}\|\bws\|_{\lis}\leq C, \\\\
			\hspace*{-0.5cm}&&\|\Delta \bv(\cdot, s_0)\|_{\lis}\leq C \quad \mathrm{and} \quad \|\Delta \bw(\cdot, s_0)\|_{\lis}\leq C.
		\end{array}
		\right.\label{l4.0}
\end{align}}
Next, we derive boundedness in $t\in (s_0,\tmax)$.

\begin{lemma}\label{l3}
Suppose that $\Omega\subset\rsn\: (n\geq 2)$, is a bounded domain with smooth boundary. Assume that $d_1>0, \mu >0$ and $M_0<1$. If $ m$ satisfies
	\begin{align*}
		&1< m < 1+\frac{2}{n},\quad n\geq 3,\\
		&1< m < 2,\quad\qquad n=2,
	\end{align*} 
then for any $1<p<\infty$, we have
	\begin{align*}
		\ints(1-M_0)\upm &\leq \frac{2d_1(p-1)}{p^2\mu}\ints\mgufpt^2+C,
	\end{align*}
	where $C>0$.
\end{lemma}
\begin{proof}
Choosing $1<p'<p+m-1$ and using the H\"older's inequality, to get
\begin{align}
	\ints\upm	\leq&\left(\ints\left(\bu^{\frac{p}{2}}\right)^{\frac{2\gamma(p+m-1)}{p}\frac{rp}{2\gamma(p+m-1)}}\right)^{\frac{2\gamma(p+m-1)}{rp}} \nonumber\\
		&\times\left(\ints\left(\bu^{\frac{p}{2}}\right)^{\togpmp\frac{p}{2(1-\gamma)(p+m-1)}\frac{2p'}{p}}\right)^{\togpmp\frac{p}{2p'}}\nonumber\\
		\leq&\:\nufpt^{\tgpmp}_{{ \mathcal{L}^{r}}(\Omega)}\:\:\nufpt^{\togpmp}_{\ltpps},\label{l3.3}
	\end{align}
		where 
	\begin{align}
		\gamma&= \frac{\frac{p}{2p'}-\frac{p}{2(p+m-1)}}{\frac{p}{2p'}-\frac{1}{r}}\in (0, 1),\label{l3.5}
	\end{align}
is the exponent from the H\"older's inequality	and $r$ satisfies
	\begin{align*}
		r=
		\left\{
		\begin{array}{rrll}
			\hspace*{-5cm}&&\frac{2n}{n-2},  &\quad \mbox{if}\quad n\geq 3,\\\\
			\hspace*{-5cm}&&\left(\frac{2(p+m-1)}{p}, \infty\right), & \quad \mbox{if}\quad n=2,\\\\
			\hspace*{-5cm}&&\infty, & \quad \mbox{if}\quad n=1.
		\end{array}
		\right.
	\end{align*}
	
	Using the Sobolev embedding theorem, we get the estimate from \eqref{l3.3} with $C_1>0$
	\begin{align}
		\ints\upm 
		\leq& C_1\Bigg(\nufpt^{\gamma}_{\lts}\nufpt^{(1-\gamma)}_{\ltpps}+\ngufpt^{\gamma}_{\lts}\nufpt^{(1-\gamma)}_{\ltpps}\Bigg)^{\frac{2(p+m-1)}{p}}\nonumber\\
		\leq& C_1\nufpt^{\tgpmp}_{\lts}\nufpt^{\togpmp}_{\ltpps}+C_1\ngufpt^{\tgpmp}_{\lts}\nufpt^{\togpmp}_{\ltpps}.\label{l3.4}
	\end{align}

	Using the Young's inequality to the second term in R.H.S of \eqref{l3.4}, we obtain
	\begin{align}
		C_1\ngufpt^{\tgpmp}_{\lts}\nufpt^{\togpmp}_{\ltpps}\leq& \:\frac{2d_1(p-1)}{p^2\mu} \Bigg(\ngufpt^{\tgpmp}_{\lts}\Bigg)^{\frac{p}{\gamma(p+m-1)}}\nonumber\\
		&+C(d_1,p,\mu)\Bigg(C_1\nufpt^{\togpmp}_{\ltpps}\Bigg)^{\frac{1}{1-\frac{\gamma(p+m-1)}{p}}}\nonumber\\\nonumber\\
		\leq&\: \frac{2d_1(p-1)}{p^2\mu} \ngufpt^2_{\lts}+C_2\nufpt^a_{\ltpps},\label{l3.6}
	\end{align}
	where $a=\togpmp\frac{1}{1-\frac{\gamma(p+m-1)}{p}}$ and $C_2>0$. 
	Take $p'>\frac{n(m-1)}{2}$ with simple computation, we ensure that
	\begin{align*}
		\frac{2\gamma(p+m-1)}{p}<& \frac{\frac{pn}{p'}+2-n}{\frac{1}{2}\left(\frac{pn}{p'}+2-n\right)}<2.
	\end{align*}
	Substituting \eqref{l3.6} in \eqref{l3.4}, we obtain
	\begin{align}
		\ints\upm\leq& \frac{2d_1(p-1)}{p^2\mu} \ngufpt^2_{\lts}+C_2\nufpt^a_{\ltpps}
		+C_1\nufpt^{\tgpmp}_{\lts}\nufpt^{\togpmp}_{\ltpps}.\label{l3.7}
	\end{align}
Now, we estimate the terms in \eqref{l3.7} as follows
	\begin{align}
		(i)\quad &\nufpt^a_{\ltpps}= \Bigg(\ints\Big(\bu^{\frac{p}{2}}\Big)^{\frac{2p'}{p}}\Bigg)^{\frac{p}{2p'}a} = \nru^{\frac{pa}{2}}_{\lpos},\label{l3.8}\\
		(ii)\quad &\nufpt^{\tgpmp}_{\lts} = \Bigg(\ints\Big(\bu^{\frac{p}{2}}\Big)^2\Bigg)^{\frac{1}{2}\tgpmp} = \nru^{\gamma(p+m-1)}_{\lps},\label{l3.9}\\
		(iii)\quad &\nufpt^{\togpmp}_{\ltpps} = \Bigg(\ints\Big(\bu^{\frac{p}{2}}\Big)^{\frac{2p'}{p}}\Bigg)^{\frac{p}{2p'}\togpmp} = \nru^{(1-\gamma)(p+m-1)}_{\lpos}.\label{l3.10}
	\end{align}
	Substituting \eqref{l3.8}-\eqref{l3.10} into \eqref{l3.7}, we get
	\begin{align}
		\ints\upm\leq \frac{2d_1(p-1)}{p^2\mu} \ngufpt^2_{\lts}+C_2\nru^{\frac{pa}{2}}_{\lpos}+C_1\nru^{\gamma(p+m-1)}_{\lps}\: \nru^{(1-\gamma)(p+m-1)}_{\lpos}.\label{l3.11}
	\end{align}
	On the other hand, using the H\"older's inequality with $1<p'<p+m-1$, we have
	\begin{align*}
		\nru_{\lpos}\leq& C_3\Bigg(\ints\bu^{\theta p'\frac{p+m-1}{\theta p'}}\Bigg)^{\frac{\theta p'}{p'(p+m-1)}} \Bigg(\ints\bu^{(1-\theta) p'\frac{1}{(1-\theta) p'}}\Bigg)^{\frac{(1-\theta) p'}{p'}}
		\leq  C_3\nru^{\theta}_{\lpms}\nru^{(1-\theta)}_{\los},
	\end{align*}
	where 
	$\theta = \frac{(p'-1)(p+m-1)}{p'(p+m-2)}\in (0,1)$ and $C_3>0$. Hence, we obtain 
	\begin{align}
		C_2\nru^{\frac{pa}{2}}_{\lpos}
		\leq C_4 \nru^{\frac{pa\theta}{2}}_{\lpms} \qquad \text{and} \qquad \nru^{(1-\gamma)(p+m-1)}_{\lpos} 
		\leq C_5\nru^{\theta(1-\gamma)(p+m-1)}_{\lpms}.\label{l3.12}
	\end{align}
	Similar way, we obtain
	\begin{align}
		\nru_{\lps}\leq& C_6\Bigg(\ints\bu^{\eta p\frac{p+m-1}{\eta p}}\Bigg)^{\frac{\eta p}{p(p+m-1)}} \Bigg(\ints\bu^{(1-\eta) p\frac{1}{(1-\eta) p}}\Bigg)^{\frac{(1-\eta) p}{p}}
		\leq  C_6\nru^{\eta }_{\lpms}\nru^{(1-\eta) }_{\los},\nonumber\\
		\leq& C_7\nru^{\eta\gamma(p+m-1)}_{\lpms}\label{l3.13}
	\end{align}
	with 
	$\eta = \frac{(p-1)(p+m-1)}{p(p+m-2)}\in (0,1)$ and $C_i>0, i=4,5,6,7$.
	Substituting \eqref{l3.12} and \eqref{l3.13} in \eqref{l3.11}, finally we achieve
	\begin{align}
		\ints\upm\leq& \frac{2d_1(p-1)}{p^2\mu} \ngufpt^2_{\lts}+C_5\nru^{\frac{pa\theta}{2}}_{\lpms}+C_8\nru^{\eta\gamma(p+m-1)}_{\lpms}\nru^{\theta(1-\gamma)(p+m-1)}_{\lpms}\nonumber\\\nonumber\\
		=& \frac{2d_1(p-1)}{p^2\mu} \ngufpt^2_{\lts}+C_5\nru^{\frac{pa\theta}{2}}_{\lpms}+C_8\nru^{[\eta\gamma+\theta(1-\gamma)](p+m-1)}_{\lpms},\label{l3.15}
	\end{align}
where $C_8>0$. Again, using the Young's inequality with $C_9, C_{10}>0$, the second term can be written as
	\begin{align}
		C_5\nru^{\frac{pa\theta}{2}}_{\lpms}\leq \frac{M_0}{2} \nru^{\frac{pa\theta}{2}\frac{2(p+m-1)}{pa\theta}}_{\lpms}+C(M_0)C_9
		\leq \frac{M_0}{2} \nru^{p+m-1}_{\lpms}+C_{10}.\label{l3.16}
	\end{align}
	By simple computations, we calculate
	\begin{align}
		\gamma <& \frac{\frac{p}{p'}-\frac{p}{p+m-1}}{\frac{p}{p'}-(2-m)}.\label{l3.17}
	\end{align}
	This guarantee that $\frac{pa\theta}{2(p+m-1)}<1$.
	Once again, using the Young's inequality with $C_{11}, C_{12}>0$ to the third term of \eqref{l3.15}, the estimates gives
	\begin{align}
		C_8\nru^{[\eta\gamma+\theta(1-\gamma)](p+m-1)}_{\lpms}\leq &\frac{M_0}{2} \left( \nru^{[\eta\gamma+\theta(1-\gamma)](p+m-1)}_{\lpms}\right)^{\frac{1}{\eta\gamma+\theta(1-\gamma)}}+C(M_0)C_{11}\nonumber\\
		\leq & \frac{M_0}{2} \nru^{p+m-1}_{\lpms}+C_{12},\label{l3.19}
	\end{align}
	where $\eta\gamma+\theta(1-\gamma)<1$ because of the definition of $\gamma$, $\theta$ and $\eta$. 
	Substituting \eqref{l3.16} and \eqref{l3.19} in \eqref{l3.15}, we obtain that
	\begin{align*}
		\ints\upm &\leq \frac{2d_1(p-1)}{p^2\mu} \ngufpt^2_{\lts}+\frac{M_0}{2} \nru^{p+m-1}_{\lpms}+\frac{M_0}{2} \nru^{p+m-1}_{\lpms}+C_{13},
	\end{align*}
	where the constant $C_{13}>0$. Finally, we obtain
	\begin{align*}
		\ints(1-M_0)\upm &\leq \frac{2d_1(p-1)}{p^2\mu}\ints\mgufpt^2+C_{13}.
	\end{align*}
	Now, we will derive the condition to the first case. 
	For $n\geq 3$ and $r=\frac{2n}{n-2}$, we have from \eqref{l3.5}
	\begin{align}
		\gamma&= \frac{\frac{pn}{2p'}-\frac{pn}{2(p+m-1)}}{\frac{pn}{2p'}+1-\frac{n}{2}}=\frac{\frac{p}{p'}-\frac{p}{p+m-1}}{\frac{p}{p'}-1+\frac{2}{n}}.\label{l3.18}
	\end{align}
	Comparing the gamma values \eqref{l3.17} with \eqref{l3.18}, we acquire
	\begin{align*}
		m< 1+\frac{2}{n}.
	\end{align*}
	Next, for the case $n=2$, we have
	\begin{align*}
		\gamma=\frac{\frac{p}{p'}-\frac{p}{p+m-1}}{\frac{p}{p'}-\frac{2}{r}}.
	\end{align*}
Now \eqref{l3.17} holds true for the case $n=2$ whenever $m<2-\frac{2}{r}$. 
This completes the proof.
\end{proof}

\begin{lemma}\label{l4}
Suppose that $\Omega\subset\rsn\: (n\geq 2)$, is a bounded domain with smooth boundary. Moreover, let $f$ fulfills \eqref{3.1} and for any $p > 1$, if $l+k < 1+\frac{2}{n}$, where $k\geq 1$ and $1<m< 1+\frac{2}{n}$, then we have
	\begin{align}
		\hspace*{3cm}\nut_{\lps}\leq C, \hspace*{3cm} \forall \, t\in(s_0,\tmax),\label{l4.1}
	\end{align}
	for some $C>0$.
\end{lemma}
\begin{proof}
	From Lemma \ref{l0}, $M_0\leq M(t)\leq 1$ for $M_0<1$. Multiply with $p\upo$, $p>1$, in the first equation of \eqref{2} and integrate over $\Omega$, we get
	\begin{align*}
		p\ints\bu_t\upo=&d_1 p\ints\upo\Delta\bu-p\chi\ints\upo\nabla\cdot\left(\buk\gv\right)+p\xi\ints\upo\nabla\cdot\left(\buk\gw\right)\\&
		+\mu p\ints\upo\bum\left(1-\ints\bu(x,t)\dx\right).
	\end{align*}
	The integration by parts, gives us that
	\begin{align}
		\dt\ints\up\leq&-d_1p(p-1)\ints\upt|\gu|^2+\chi p(p-1)\ints\bu^{p+k-2}\gu\cdot\gv\nonumber\\
		&-\xi p(p-1)\ints\bu^{p+k-2}\gu\cdot\gw+\mu p\ints\upm-\mu p M_0\ints\upm.\label{l4.1.1}
	\end{align}
	Again use of integration by parts to the second and third term in R.H.S and using Lemma \ref{l3}, leads to 
	\begin{align}
		\dt\ints\up\leq&\frac{-4 d_1 (p-1)}{p}\ints\mgufpt^2-\chi p\fpk\ints\upk\lv+\xi p\fpk\ints\upk\lw\nonumber\\
		&+\frac{2d_1 (p-1)}{p}\ints\mgufpt^2+C_0\nonumber\\
		\leq&\frac{-2 d_1 (p-1)}{p}\ints\mgufpt^2+I_1+I_2+C_0.\label{l4.2}
	\end{align}
	Let $\delta=\frac{np+2}{(p+k-1)n}>1,$ now applying the Young's inequality with $\epsilon>0$, we get
	\begin{align}
		I_1=\left|-\chi p\fpk\ints\upk\lv\right|\leq& \epsilon \ints \bu^{(p+k-1)\delta}+C_1\ints\mlv^\frac{\delta}{\delta-1}\nonumber\\
		\leq& \epsilon \ints \bu^\fnpt+C_1\ints\mlv^\fnpk,\label{l4.3}
	\end{align}
	where $C_1>0$. 
	Similarly for $I_2$
	\begin{align}
		I_2=\xi p\fpk\ints\upk\lw\leq& \epsilon \ints \bu^{(p+k-1)\delta}+C_2\ints\mlw^\frac{\delta}{\delta-1}\nonumber\\
		\leq& \epsilon \ints \bu^\fnpt+C_2\ints\mlw^\fnpk,\label{l4.4}
	\end{align}
	where $C_2>0$. 
	Substituting. \eqref{l4.3} and \eqref{l4.4} in \eqref{l4.2}, we obtain
	\begin{align}
		\dt\ints\up+\frac{d_1 (p-1)}{p}\ints\mgufpt^2\leq&\frac{-d_1 (p-1)}{p}\ints\mgufpt^2+2\epsilon \ints \bu^\fnpt\nonumber\\
		&+C_1\ints\mlv^\fnpk+C_2\ints\mlw^\fnpk+C_0.\label{l4.6}
	\end{align}
	Using the Gagliardo-Nirenberg inequity and the Young's inequality, we achieve
	\begin{align*}
		\ints\up&=\left\|\bu^{\frac{p}{2}}\right\|^2_{\lts}\leq C_3\left(\left\|\nabla\bu^{\frac{p}{2}}\right\|^{2b}_{\lts}\,\,\left\|\bu^{\frac{p}{2}}\right\|^{2(1-b)}_{\ltps}+\left\|\bu^{\frac{p}{2}}\right\|^2_{\ltps}\right)\nonumber\\
		&\leq \epsilon\left(\left\|\nabla\bu^{\frac{p}{2}}\right\|^{2b}_{\lts}\right)^\frac{1}{b}+C_3\,C(\epsilon)\left(\left\|\bu^{\frac{p}{2}}\right\|^{2(1-b)}_{\ltps}\right)^\frac{1}{1-b}+C_3\left\|\bu^{\frac{p}{2}}\right\|^2_{\ltps}\nonumber\\
		&\leq \epsilon\left\|\nabla\bu^{\frac{p}{2}}\right\|^{2}_{\lts}+C_{4} \left\|\bu^{\frac{p}{2}}\right\|^2_{\ltps}\nonumber\\
		&\leq \epsilon\left\|\nabla\bu^{\frac{p}{2}}\right\|^{2}_{\lts}+C_{4} \nru^p_{\los}\nonumber\\
		&\leq \frac{2-(k-1)n}{np+2}\frac{d_1(p-1)}{p}\left\|\nabla\bu^{\frac{p}{2}}\right\|^{2}_{\lts}+C_{5}\nonumber\\
		&\leq \frac{2-(k-1)n}{np+2}\frac{d_1(p-1)}{p}\ints\left|\nabla\bu^{\frac{p}{2}}\right|^{2}+C_{5},
	\end{align*}
	where $b=\frac{\frac{p}{2}-\frac{1}{2}}{\frac{p}{2}+\frac{1}{n}-\frac{1}{2}}\in(0,1)$ and $C_5>0$. Now, we can estimate the second term in \eqref{l4.6} using the above inequality as follows
	\begin{align}
		\fnpk\ints\up&\leq\frac{d_1(p-1)}{p}\ints\left|\nabla\bu^{\frac{p}{2}}\right|^{2}+C_{6},\label{l4.9}
	\end{align}
	where $C_6>0$. Again, using the Gagliardo-Nirenberg inequity, to get
	\begin{align*}
		\ints\bu^\fnpt=&\left\|\bu^{\frac{p}{2}}\right\|^{\frac{2(np+2)}{np}}_{\frac{2(np+2)}{np}}\leq C_7\left(\left\|\nabla\bu^{\frac{p}{2}}\right\|^{\frac{np}{np+2}}_{\lts}\,\,\left\|\bu^{\frac{p}{2}}\right\|^{\frac{2}{np+2}}_{\ltps}+\left\|\bu^{\frac{p}{2}}\right\|_{\ltps}\right)^{\frac{2(np+2)}{np}}\nonumber\\
		\leq& C_7\left\|\nabla\bu^{\frac{p}{2}}\right\|^{2}_{\lts}\,\,\left\|\bu^{\frac{p}{2}}\right\|^{\frac{4}{np}}_{\ltps}+C_7\left\|\bu^{\frac{p}{2}}\right\|^{\frac{2(np+2)}{np}}_{\ltps}\nonumber\\
		\leq& C_7\left\|\nabla\bu^{\frac{p}{2}}\right\|^{2}_{\lts}\:\nru^{\frac{2}{n}}_{\los}+C_7\nru^{\frac{np+2}{n}}_{\los}\nonumber\\
		\leq& C_8\left\|\nabla\bu^{\frac{p}{2}}\right\|^{2}_{\lts}+C_9\nonumber\\
		\leq& C_8\ints\left|\nabla\bu^{\frac{p}{2}}\right|^{2}+C_9,\nonumber
	\end{align*}
where $C_8, C_9>0$	Now, we can estimate the third term in \eqref{l4.6} using the above inequality as follows
	\begin{align}
		\frac{-d_1(p-1)}{p}\ints\left|\nabla\bu^{\frac{p}{2}}\right|^{2}\leq \frac{-d_1(p-1)}{pC_8}\ints\bu^\fnpt+C_{10},\label{l4.7}
	\end{align}
where $C_{10}>0$. Substituting \eqref{l4.9} and \eqref{l4.7} in \eqref{l4.6}, we arrive at
		\begin{align}
			\dt\ints\up+\fnpk\ints\up\leq&\frac{-d_1(p-1)}{pC_8}\ints\bu^\fnpt+2\epsilon \ints \bu^\fnpt\nonumber\\
			&+C_1\ints\mlv^\fnpk+C_2\ints\mlw^\fnpk+C_{11},\label{l4.12}
		\end{align}
	where $C_{11}>0$. Applying the variation-of-constants formula to \eqref{l4.12}, we obtain the following estimate with $\epsilon=\frac{d_1(p-1)}{4pC_8}$
		\begin{align}
			\ints\up\leq&\frac{-d_1(p-1)}{2pC_8}\intts\enpkts\ints\bu^\fnpt+C_1\intts\enpkts\ints\mlv^\fnpk\nonumber\\
			&+C_2\intts\enpkts\ints\mlw^\fnpk+C_{12},\label{l4.13}
		\end{align}
		where $C_{12}>0$. 
		From  Lemma \ref{l2}, there
		exists $C_p > 0$ such that
		\begin{align}
			C_1\intts\ints\enpkts\mlv^\fnpk\leq& C_{13}\intts\ints\enpkts|\bu|^{\frac{(np+2)l}{2-(k-1)n}}\nonumber\\
			&+C_{14}\Big\|\bv(\cdot, s_0)\Big\|^{\fnpk}_{{ \mathcal{W}}^{2,\fnpk}}.\label{l4.14}
		\end{align}
		By the similar manner, we have
		\begin{align}
			C_2\intts\ints\enpkts\mlw^\fnpk\leq& C_{15}\intts\ints\enpkts|\bu|^{\frac{(np+2)l}{2-(k-1)n}}\nonumber\\
			&+C_{16}\Big\|\bw(\cdot, s_0)\Big\|^{\fnpk}_{{ \mathcal{W}}^{2,\fnpk}}.\label{l4.15}
		\end{align}
		Substituting \eqref{l4.14} and \eqref{l4.15} in \eqref{l4.13}, we obtain that
		\begin{align}
			\ints\up\leq&\frac{-d_1(p-1)}{2pC_8}\intts\ints\enpkts\bu^\fnpt\nonumber\\
			&+C_{17}\intts\ints\enpkts\bu^{\frac{(np+2)l}{2-(k-1)n}}+C_{18},\label{l4.16}
		\end{align}
		where  $C_{18}>0$. 
		Using the Young's inequality with $l+k<1+\frac{2}{n}$, we have
		\begin{align}
			C_{17}\ints\bu^{\frac{(np+2)l}{2-(k-1)n}}\leq \frac{d_1(p-1)}{2pC_8} \ints \bu^\fnpt+C_{19},\label{l4.17}
		\end{align}
		where  $C_{19}>0$. 
		Substituting \eqref{l4.17} in \eqref{l4.16}, finally, we arrive at
		\begin{align*}
			\ints\up\leq C_{20},\qquad\qquad\forall t\in(s_0, \tmax).
		\end{align*}
		where the constant $C_{20}>0$.\\
		Now, for $M_0\geq 1$, from Lemma \ref{l0} we have $1\leq M(t)\leq M_0$. Then \eqref{l4.1.1} can be written as
		\begin{align*}
			\dt\ints\up\leq&-d_1p(p-1)\ints\upt|\gu|^2+\chi p(p-1)\ints\bu^{p+k-2}\gu\cdot\gv\\
			&-\xi p(p-1)\ints\bu^{p+k-2}\gu\cdot\gw.
		\end{align*}
		Applying the same procedure, one can obtain the boundedness. This completes the proof.
	\end{proof}

{\bf Proof of Theorem \ref{t1}.} We proceed by contradiction. Assume that $\tmax < \infty$. By Lemma \ref{l5}, the bounds remain valid for all $t \in (0,\tmax)$. However, the blow-up criterion \eqref{l1.1} contradicts this result. Hence, the maximal existence time cannot be finite i.e, $\tmax=\infty$. Consequently, we obtain the uniform estimates
	\begin{align*}
	\nut_{\lis}+\nvt_{\wsq}+\nwt_{\wsq}\leq C,
\end{align*}
for all $t > 0$ and the constant $C$ is positive. This completes the proof.

\section{Conclusion}
\quad This work analyzes the existence of classical solution to the attraction repulsion chemotaxis system with a nonlocal source and sublinear production. For $n\geq 2$, it is proved that the system \eqref{2} has a bounded classical solution when the parameters $k, l$ and $m$ satisfy some suitable conditions.



\begin{thebibliography}{100}
	\bibitem{nbellomo}N. Bellomo, A. Bellouquid, Y. Tao, M. Winkler, Toward a mathematical theory of	Keller–Segel models of pattern formation in biological tissues, Math. Models Methods Appl. Sci., 25 (2015), 1663–1763
	
	\bibitem{sbian2017} S. Bian, Global solutions to a nonlocal Fisher-KPP type problem, Acta Appl. Math., 147 (2017), 187-195.
	
	\bibitem{sbian}S. Bian, L. Chen, E. A. Latos, Global existence and asymptotic behavior of solutions to a 	nonlocal Fisher-KPP type problem, Nonlinear Anal., 149 (2017) 165-176.
	
	\bibitem{bian2018} S. Bian, L. Chen, E.A. Latos, Nonlocal nonlinear reaction preventing blow-up in supercritical case of chemotaxis system, Nonlinear Anal., 176 (2018), 178-191.
	
	\bibitem{xcao} X. Cao, Boundedness in a quasilinear parabolic-parabolic Keller-Segel system with logistic source, J. Math. Anal. Appl., 412 (2014),  181-188.
		
	\bibitem{chiyo}Y. Chiyo, M. Mizukami, T. Yokota, Global existence and boundedness in a fully parabolic attraction-repulsion chemotaxis system with signal-dependent sensitivities without logistic source, J. Math. Anal. Appl., 489 (2020), 124153.
	
	\bibitem{chiyo2024} Y. Chiyo, F.G. D\"uzg\"un, S. Frassu, G. Viglialoro, Boundedness through nonlocal dampening effects in a fully parabolic chemotaxis model with sub and superquadratic growth, Appl. Math. Optim., 89 (2024), 9.
		
	\bibitem{columbu20242} A. Columbu, S. Frassu, G. Viglialoro, Refined criteria toward boundedness in an attraction–repulsion chemotaxis system with nonlinear productions, Appl. Anal., 103 (2023), 415–431.
	
	\bibitem{columbu2024} A. Columbu, R.D. Fuentes, S. Frassu, Uniform-in-time boundedness in a class of local and nonlocal nonlinear attraction–repulsion chemotaxis models with logistics, Nonlinear Anal. RWA., 79 (2024), 104135.
	
	\bibitem{wdu2024} W. Du, A further study on an attraction-repulsion chemotaxis system with logistic source, AIMS Math., 9 (2024), 16924–16930.
	
		\bibitem{evans} L. C. Evans, \emph{Partial Differential Equations,} American Mathematical Society, (1998), ISBN 9780821848593.
	
	\bibitem{frassu} S. Frassu, G. Viglialoro, Boundedness for a fully parabolic Keller--Segel model with sublinear segregation and superlinear aggregation, Acta Appl. Math., 171 (2021), 19.
	
	\bibitem{funtes2025} R.D. Fuentes, S. Frassu, G. Viglialoro, Dissipation Through Combinations of Nonlocal and Gradient Nonlinearities in Chemotaxis Models, Acta Appl. Math., 195 (2025), 10.
	
	\bibitem{qguo} Q. Guo, Z. Jiang, S. Zheng, Critical mass for an attraction-repulsion chemotaxis system, Appl. Anal., 97 (2018), 2349-2354.
	
	\bibitem{hetian}X. He, M. Tian, S. Zheng, Large time behavior of solutions to a quasilinear attraction-repulsion chemotaxis system with logistic source, Nonlinear Anal. RWA., 54 (2020), 103095.
	
	\bibitem{hieber} M. Hieber, J. Pruss, Heat kernels and maximal $\lps$-$\lqs$ estimate for parabolic evolution equations, Commun. Partial Differ. Equ., 22 (1997), 1647-1669.
	
	\bibitem{lhong} L. Hong, M. Tian, S Zheng, An attraction-repulsion chemotaxis system with nonlinear productions, J. Math. Anal. Appl., 484 (2020), 123703.
	
	\bibitem{horstmann}D. Horstmann, From 1970 until present: the Keller–Segel model in chemotaxis and its consequences I, Jahresberichte DMV., 105 (2003), 103–165.
	
	\bibitem{horst} D. Horstmann, M. Winkler, Boundedness vs. blow-up in a chemotaxis system, J. Differ. Equ., 215 (2005), 52-107.
	
	\bibitem{rhu2024} R. Hu, P. Zheng, Global Stability in a Two-species Attraction–Repulsion System with Competitive and Nonlocal Kinetics, J. Dynam. Differ. Equ., 36 (2024), 2555–2592.
	
	\bibitem{hyjin}H.Y. Jin, Boundedness of the attraction-repulsion Keller-Segel system, J. Math. Anal. Appl., 422 (2015), 1463--1478.
	
	\bibitem{jinliu} H.Y. Jin, Z. Liu, Large time behavior of the full attraction--repulsion Keller--Segel system in the whole space, Appl. Math. Lett., 47 (2015), 13-20.
	
	\bibitem{jinwang} H.Y. Jin, Z.A. Wang, Asymptotic dynamics of the one-dimensional attraction--repulsion Keller-Segel model, J. Math. Methods Appl. Sci., 38 (2015), 444-457.
	
	\bibitem{hyjinwang} H.Y. Jin, Z.A. Wang, Global stabilization of the full attraction-repulsion Keller--Segel system, Discrete Continuous Dyn. Syst. Ser. S., 40 (2020),  3509-3527.
	
	\bibitem{lan} J. Lankeit, M. Winkler, Facing low regularity in chemotaxis systems, Jber. DMV., 122 (2020), 35-64.
	
	\bibitem{limu} D. Li, C. Mu, K. Lin, L. Wang, Large time behavior of solution to an attraction-repulsion chemotaxis system with logistic source in three dimensions, J. Math. Anal. Appl., 448 (2017), 914-936.
	
	\bibitem{liwang} Y. Li, W. Wang, Boundedness in a four-dimensional attraction-repulsion chemotaxis system with logistic source, Math. Methods Appl. Sci., 41 (2018), 4936-4942.
	
	\bibitem{lixiang} X. Li, Z. Xiang, On an attraction-repulsion chemotaxis system with a logistic source, IMA J. Appl. Math., 81 (2016), 165-198.
	
	\bibitem{linmu} K. Lin, C. Mu, Global existence and convergence to steady states for an attraction-repulsion chemotaxis system, Nonlinear Anal. RWA., 21 (2016), 630-642.
	
	\bibitem{klin} K. Lin, C. Mu, Y. Gao, Boundedness and blow up in the higher-dimensional attraction-repulsion chemotaxis system with nonlinear diffusion, J. Differ. Equ., 261 (2016), 4524-4572.
	
	\bibitem{liutao} D. Liu, Y. Tao, Global boundedness in a fully parabolic attraction-repulsion chemotaxis model, J. Math. Methods Appl. Sci., 38 (2015),  2537-2546.
	
	\bibitem{liuwang}J. Liu, Z.A. Wang, Classical solutions and steady states of an attraction-repulsion chemotaxis model in one dimension, J. Biol. Dyn., 6 (2012), 31-41.
	
	\bibitem{mliu} M. Liu, Y. Li, Finite-time blowup in attraction-repulsion systems with nonlinear signal production, Nonlinear Anal. RWA., 61 (2021),  103305.
	
	\bibitem{mluca} M. Luca, Chemotactic signaling, microglia, and Alzheimer's disease senile plaques: Is there a connection?, Bull. Math. Biol., 65 (2003),  693-730.
	
	\bibitem{mnegreanu2021} M. Negreanu, J.I. Tello, A.M. Vargas, On a fully parabolic chemotaxis system with nonlocal growth term, Nonlinear Anal., 213 (2021), 112518.
	
	\bibitem{gren2022}G. Ren, Global boundedness and asymptotic behavior in an attraction–repulsion chemotaxis system with nonlocal terms, Z. Angew. Math. Phys., 73 (2022), 200.
	
	\bibitem{renliu} G. Ren, B. Liu, Global boundedness and asymptotic behavior in a quasilinear attraction-repulsion chemotaxis model with nonlinear signal production and logistic-type source, Math. Models Methods Appl. Sci., 30 (2020),  2619-2689.
	
	\bibitem{gren} G. Ren, B. Liu, Large time behavior of solutions to a quasilinear attraction-repulsion chemotaxis model with nonlinear secretion, J. Math. Phys., 62 (2021), 091510.
	
	\bibitem{shiliu} S. Shi, Z. Liu, H.Y. Jin, Boundedness and large time behavior of an attraction-repulsion chemotaxis model with logistic source, Kinetic \& Related Models, 10 (2017), 855-878.
	
	\bibitem{ywang2015}Y. Wang, A quasilinear attraction-repulsion chemotaxis system of parabolic-elliptic type with logistic source, J. Math. Anal. Appl., 441 (2016), 259-292.
	
	\bibitem{ywang2016} Y. Wang, Global existence and boundedness in a quasilinear attraction-repulsion chemotaxis system of parabolic-elliptic type, Boundary Value Problems, 9 (2016).
	
	\bibitem{wangxiang} Y. Wang, Z. Xiang, Boundedness in a quasilinear 2D parabolic-parabolic attraction-repulsion chemotaxis system, Discrete Continuous Dyn. Syst. Ser. B., 21 (2016), 1953-1973.
	
	\bibitem{winkler} M. Winkler, Aggregation vs. global diffusive behavior in the higher-dimensional Keller-Segel model, J. Differ. Equ., 248 (2010), 2889-2905.
	
	\bibitem{swu} S. Wu, J. Shi, B. Wu, Global existence of solutions to an attraction-repulsion chemotaxis model with growth, Commun. Pure Appl. Anal., 16 (2017), 1037-1058.
	
	\bibitem{xuzheng} P. Xu, S. Zheng, Global boundedness in an attraction-repulsion chemotaxis system with logistic source, Appl. Math. Lett., 83 (2018), 1-6.
	
	\bibitem{lyan} L. Yan, Z. Yang, Global Existence and Blow-up of Classical Solution for an Attraction-repulsion Chemotaxis System with Logistic Source, J. adv. math. comput. sci., 30 (2019), 1-16.
	
	\bibitem{hyi2025} H. Yi, S. Qiu, G. Xu, Boundedness and asymptotic behavior in the higher dimensional fully parabolic attraction-repulsion chemotaxis system with nonlinear diffusion, Journal of Mathematical Analysis and Applications, J. Math. Anal. Appl., 541 (2025), 128709. 
	
	\bibitem{yzeng} Y. Zeng, Existence of global bounded classical solution to a quasilinear attraction-repulsion chemotaxis system with logistic source, Nonlinear Anal., 161 (2017), 182-197.
	
	\bibitem{qzhang} Q. Zhang, Y. Li, An attraction-repulsion chemotaxis system with logistic source, Z. Angew. Math. Mech., 96 (2016), 570-584.
	
	\bibitem{jzhao} J. Zhao, C. Mu, D. Zhou, K. Lin, A parabolic-elliptic-elliptic attraction-repulsion chemotaxis system with logistic source, J. Math. Anal. Appl., 455 (2017), 650-679.
	
	\bibitem{zhengmu} P. Zheng, C. Mu, X. Hu, Boundedness in the higher dimensional attraction-repulsion chemotaxis-growth system, Comput. Math. Appl., 72 (2016),  2194-2202.
	
	\bibitem{xzhou} X. Zhou, Z. Li, J. Zhao, Asymptotic behavior in an attraction-repulsion chemotaxis system with nonlinear productions, J. Math. Anal. Appl., 507 (2022), 125763.
\end{thebibliography}
\end{document}